\documentclass{article}

\usepackage[utf8]{inputenc}
\usepackage[english]{babel}
\usepackage{fullpage}
\usepackage{mathtools}

\def\lc{\left\lceil}   
\def\rc{\right\rceil}
\def\lf{\left\lfloor}   
\def\rf{\right\rfloor}
\DeclareMathOperator{\prob}{prob}

\usepackage{amsmath,amsthm,amssymb}
\usepackage{enumerate}

\newtheorem{theorem}{\noindent Theorem}
\newtheorem{lemma}{\noindent Lemma}
\newtheorem{proposition}{\noindent Proposition}
\newtheorem{corollary}{\noindent Corollary}

\title{Regular behaviour of the maximal 
hypergraph chromatic number}

\author{Danila Cherkashin\footnote{National Research University Higher School of Economics, Soyuza Pechatnikov str., 16, St. Petersburg, Russian Federation}, 
Fedor Petrov\footnote{Saint Petersburg State University, Faculty of Mathematics and Mechanics; St.~Petersburg Department of V.~A.~Steklov Institute of Mathematics of the Russian Academy of Sciences.}}

\begin{document}

\maketitle

\begin{abstract}
Let $m(n,r)$ denote the minimal number of edges in an $n$-uniform hypergraph which is not $r$-colorable. It is known that for a fixed $n$ one has
\[
c_n r^n < m(n,r) < C_n r^n. 
\]
We prove that for any fixed $n$ the sequence $a_r := m(n,r)/r^n$ has a limit, which was conjectured by Alon.
We also prove the list colorings analogue of this 
statement.
\end{abstract}

\section{Introduction}

A hypergraph $H=(V,E)$ consists of a finite set of \textit{vertices}
$V$ and a family $E$ of the subsets of $V$,
which are called \textit{edges}. A hypergraph is called
\textit{$n$-uniform} if every edge has size $n$.
A \textit{vertex $r$-coloring} of a hypergraph $H = (V,E)$ is a map from $V$ to $\{1,\ldots,r\}$.
A coloring is \textit{proper} if there is no monochromatic edges, i.e., any
edge $e\in E$ contains two vertices
of different color.
The \textit{chromatic number} of a hypergraph $H$ is the smallest number $\chi(H)$ such that there exists a proper $\chi(H)$-coloring of $H$.   
Let $m(n,r)$ be the minimal number of edges in an $n$-uniform hypergraph with chromatic number more than $r$.

We are interested in the case when $n$ is much smaller than $r$ (see~\cite{RaiSh} for the general case and related problems).

\subsection{Upper bounds}

For $n=2$ (i.e. for graphs) the problem of finding $m(n,r)$ is trivial.
Indeed, $m(2,r) \geq \binom{r+1}{2}$ since any coloring of a given $G$ in $\chi(G)$ colors 
should contain an edge between every pair of colors, otherwise one can join them. From the other hand, the complete graph on $r+1$ vertices gives an example.

Erd{\H o}s conjectured~\cite{erdos1979some} that 
\[
m(n,r) = \binom{(n-1)r+1}{n},
\]
for $r > r_0(n)$, that is achieved by the complete hypergraph on $(n-1)r+1$ vertices.

However Alon~\cite{alon1985hypergraphs} disproved the conjecture for $n$ large enough by using the
estimate
\[
m(n,r) \leqslant \min_{a \geq 0} T(r(n + a - 1) + 1, n + a, n),
\]
where the Tur{\'a}n number $T(v,k,n)$ is the smallest number of edges in an $n$-uniform hypergraph on $v$ vertices
such that every induced subgraph on $k$ vertices contains an edge. Different bounds on Tur{\'a}n numbers beat the complete hypergraph construction when $n > 3$ (see~\cite{sidor} for a survey).
So the case $n = 3$ is in some sense the most 
interesting.

Using the same inequality with better bounds on Tur{\'a}n numbers, Akolzin and Shabanov~\cite{akolzin2016colorings} showed that
\[
m(n,r) < Cn^3\ln n \cdot r^n.
\]

Alon~\cite{alon1985hypergraphs} conjectured that for a fixed $n$ the quantity $m(n,r)$ has regular behavior, i.e. the sequence $m(n,r)/r^n$ has a limit.

%Перенес в другой раздел

\subsection{Lower bounds}

There are several ways to show an inequality of type $m(n,r) > c(n) r^n$.
Alon~\cite{alon1985hypergraphs} uses an alteration-type trick to get the first bound of such type
\[
m(n,r) \geq (n-1) \lc \frac{r}{n} \rc \lf \frac{n-1}{n}r \rf ^{n-1}.
\]
Pluh{\'a}r's random greedy approach~\cite{Pl} gives the bound
\[
m(n,r) > c\sqrt{n}r^n
\]
as noted in~\cite{RaiSh}.
Finally, combining two previous arguments Akolzin and Shabanov~\cite{akolzin2016colorings} proved that
\[
m(n,r) > c\frac{n}{\ln n}r^n.
\]

\subsection{List colorings}

%У Судакова константа не зависит от n, см. обзор.
Let $H = (V,E)$ be a hypergraph and let $\{L(v)\}$, $v \in V(H)$ be sets; we refer to these sets as \textit{lists}. A \textit{list coloring} of $H$ is an assignment of a color from $L(v)$ to each 
$v \in V (H)$.
The \textit{list chromatic number} of a hypergraph $H$ 
Define the quantity $m_c(n,r)$ as the minimal number of edges 
of an $n$-uniform hypergraph with list
chromatic number greater than $r$.

By definition, $m_c(n,r) \leq m(n,r)$, and this is the only known upper bound on $m_c(n,r)$ (also, it is not known whether $m_c(n,r) = m(n,r)$ for all $n$, $r$).

It was recently proved by B.~Sudakov (unpublished) that there is $c > 0$ such that 
\[
m_c(n,r) \geq cr^n
\]
for all $n$, $r > r_0(n)$.

\paragraph{Structure of the paper.} In Section 2 contains the proof of Alon conjecture that the sequence $a_r := m(n,r)/r^n$ has a limit. Section 3 proofs the same result for $m_c(n,r)$.
The final section consists of open questions.

\section{Colorings}

Fix $n>1$ and denote by
$f(N)$ the maximal possible chromatic
number of an $n$-uniform
hypergraph with $N$ edges,
in particular
$f(0)=1$.
The function $f:
\mathbb{Z}_{\geqslant 0}\rightarrow \mathbb{R}_{\geqslant 1}$ non-strictly increases and 
satisfies
$$m(n,r)=\min\{N:f(N)>r\}.$$
Therefore $m(n,r)\sim C r^n$ if and only if
$f(N)\sim (N/C)^{1/n}$. 

Here is the crucial

\begin{lemma}\label{cru}
For any $N > 0$ and any positive integer $p$ we have 
\begin{equation}\label{recursive}
  f(N)\leqslant \max_{a_1+a_2+\dots +a_p\leqslant N/p^{n-1}} f(a_1)+f(a_2)+\dots+f(a_p).
\end{equation}

\end{lemma}

\begin{proof}
Let $H=(V,E)$ be an $n$-uniform hypergraph with $|E|=N$. 

Choose the auxiliary colors $\eta(v)\in \{1,2,\dots,p\}$
at random uniformly and independently and denote $V_i=\eta^{-1}(\{i\})$.
Let $H_i = (V_i,E_i)$ be the hypergraph induced by $H$ on $V_i$.
The expectation of $\sum_i |E_i|$ equals $|E|/p^{n-1}$ because each edge of $H$ belongs to some $H_i$ with the same probability $1/p^{n-1}$.
Therefore there exists a certain auxiliary coloring $\eta$ such that 
\[
\sum |E_i| \leqslant N/p^{n-1}.
\]
Fix such a coloring $\eta$ and properly color each $H_i$ using $f(|E_i|)$ colors,
using disjoint sets of colors for different $i$.
Totally we use $\sum f(|E_i|)$ colors
and $H$ is colored properly.

Since $H$ was an arbitrary $n$-uniform hypergraph with $N$ the proof is completed.
\end{proof}

Further part of the proof is
completely analytical, all combinatorics was
in Lemma \ref{cru}.

Namely, the following general statement holds:

\begin{theorem}\label{gt}
Assume that $n>1$ is a fixed 
integer, $N_0>0$ is a constant,
$f:\mathbb{Z}_{\geqslant 0}\to \mathbb{R}_{>0}$
is a function satisfying \eqref{recursive} for
all $N\geqslant N_0$ and $p\in \{2,3\}$. 
Then
\[
\lim_{x\to \infty} \frac{f(x)}{x^{1/n}}
\]
exists and is finite.
\end{theorem}

For proving Theorem~\ref{gt} we use
the followings Lemmata.

\begin{lemma}\label{segment}
Denote $c_n=\lceil
(1-2^{1/n-1})^{-n}\rceil$.
For any
$M\geqslant N_0$ the inequality
\[
f(N)\leqslant 
N^{1/n}\cdot \max_{M\leqslant a< c_n M} f(a)\cdot a^{-1/n}
\]
holds for all $N\geqslant M$.
\end{lemma}

\begin{proof}
Induct on $N\in \{M,M+1,\ldots\}$. The base 
$N<c_n M$ is clear. 

The induction step
from  $M,M+1,\dots,N-1$ to  
$N$ assuming $N\geqslant c_n M$. 

Denote
\[
\lambda=\max_{M\leqslant a< c_n M} f(a)\cdot a^{-1/n}.
\]
By \eqref{recursive} with $p=2$
we have $f(N)\leqslant f(a)+f(b)$
for certain 
non-negative 
integers $a,b$ such that $a+b\leqslant N/2^{n-1}$. If $\min(a,b)\geqslant M$, then
by induction proposition we get
\[
f(a)+f(b)\leqslant \lambda (a^{1/n}+b^{1/n})\leqslant 2\lambda \left (\frac{a+b}2 \right )^{1/n}\leqslant \lambda N^{1/n},
\]
as desired. If, for example, $a<M$, 
we get
\[
f(a)+f(b)\leqslant f(M)+f(b) \leqslant \lambda \left (M^{1/n}+\left (\frac{N}{2^{n-1}} \right)^{1/n} \right)
\leqslant \lambda N^{1/n}
\]
provided that  $N\geqslant c_n M$.
\end{proof}

Lemma~\ref{segment} in particular implies that the maxima
$M(k)$
of the function $g(x) := f(x)x^{-1/n}$ 
over the segments $[c_n^k,c_n^{k+1}]$ eventually
(for $k\geqslant k_0$)
do not increase.
Let $\alpha_0$ denote the 
limit of $M(k)$, it is also the upper
limit of the function $g$.

Fix $p$ in Lemma~\ref{cru}.

Further we need the following standard technical
\begin{proposition}\label{Jensen_acc}
For any $\theta>1$ there exists
$\delta>0$ such that for
all non-negative real numbers $x_1,\ldots,x_p$
with the arithmetic mean $x_0=(x_1+\ldots+x_p)/p$ the 
inequality
\[
\sum_{i=1}^p x_i^{1/n}\geqslant (p-\delta)\cdot x_0^{1/n}
\]
yields $x_i\in [x_0/\theta,x_0\cdot \theta]$. 
\end{proposition}

\begin{proof}
The case $x_0=0$ is clear. If $x_0>0$,
denote $y_i=x_i/x_0$, then $\sum y_i=p$
and $\sum y_i^{1/n}\geqslant p-\delta$.
Let $\ell(x)=1+(x-1)/n$ be a tangent
line to the graph of the function $x^{1/n}$ at point
$(1,1)$. We have $\sum \ell(y_i)=p$. By concavity we have $y^{1/n}\leqslant \ell(y)$
with equality only at $y=1$, and for
given $\theta>1$ there exists $\delta>0$
such that $\ell(y)-y^{1/n}>\delta$ for
$y\notin [1/\theta,\theta]$. Therefore
$$
\delta\geqslant p-\sum_{i=1}^p y_i^{1/n}=\sum_{i=1}^p \left(\ell(y_i)-y_i^{1/p}\right),
$$
all summands $\ell(y_i)-y_i^{1/p}$ belong
to $[0,\delta]$ and therefore $y_i\in [1/\theta,\theta]$
and $x_i\in [x_0/\theta,x_0\theta]$.
\end{proof}

We proceed with the proof of Theorem~\ref{gt}.

Let $N$ be a large integer with $g(N)=\alpha_0+o(1)$.
In other words, $N$ grows to infinity along
such a subsequence that $g(N)$ converges to $\alpha_0$. 
Find for this $N$ the numbers
$a_1,\dots,a_p$ as in Lemma~\ref{cru}.
Note that for any
$\varepsilon>0$ there exists $C>0$ such that 
$f(a)\leqslant (\alpha_0+\varepsilon)a^{1/n}+C$ for all integers $a\geqslant 0$. It follows that $f(a)\leqslant \alpha_0 a^{1/n}+o(N^{1/n})$ uniformly for all $a\leqslant N$. 
Therefore
$$
\alpha_0\cdot p\cdot \left(\frac{a_1+\ldots+a_p}p\right)^{1/n}
\leqslant
\alpha_0N^{1/n}=f(N)+o(N^{1/n})\leqslant 
\alpha_0\sum_{i=1}^p a_i^{1/n}+o(N^{1/n}).
$$
So all inequalities here are
equalities with accuracy $o(N^{1/n})$.
In particular $\sum a_i=N/p^{n-1}+o(N)$ and 
all $a_i$ are asymptotically equal to $N/p^n+o(N)$ by Proposition 
\ref{Jensen_acc}. Also $f(a_i)=\alpha_0 N^{1/n}/p+o(N^{1/n})$ for all $i=1,\ldots,p$.
Equivalently, $g(a_i)=\alpha_0+o(1)$ for all 
$i=1,\ldots,p$.

Consider the numbers of the form
$2^{nx} 3^{ny}$ with non-negative
integer $x,y$, call them \textit{appropriate}
numbers.

So we proved that for large $N$ with 
$g(N)=\alpha_0+o(1)$
there exists
$\tilde{N}=N/p^{n}+o(N)$ with 
$g(\tilde{N})=\alpha_0+o(1)$. 
Consecutively using this for $p\in \{2,3\}$ we conclude that whenever $g(N)=\alpha_0+o(1)$ and $R$ is appropriate, then there exists $a=N/R+o(N)$ such that
$g(a)=\alpha_0+o(1)$.

The ratio of two consecutive appropriate numbers tends to 1 by the basic
Dirichlet--Kronecker
Diophantine approximation lemma.
Fix $\rho>1$ and choose appropriate
numbers 
$r_1<r_2<\dots<r_m$ so that $r_{i+1}/r_i<\rho$,
but $r_1<c_n^S$, $r_m>c_n^{S+10}$ for certain positive integer $S$.

So we may find numbers
$N_i=N/r_i+o(N)$ such that $g(N_i)=\alpha_0+o(1)$ 
for all $i=1,2,\dots,m$.

For large $k$ choose $N\in [c_n^k,c_n^{k+1}]$ 
with maximal possible value
$g(N)$; we have $g(N)=\alpha_0+o(1)$.
For any integer number $x$ in the
segment 
$[c_n^{k-S-2},c_n^{k-S-1}]$
choose minimal $i$ such that
$x>N_i$. Then $x\leqslant N_i\cdot \rho$ and
\[
f(x)\geqslant f(N_i)=(\alpha_0+o(1)) N_i^{1/p} \geqslant (\alpha_0+o(1)) (x/\rho)^{1/n}.
\]
Therefore 
\[
\liminf f(x)x^{-1/n}\geqslant \alpha_0 \rho^{-1/n},
\]
and since $\rho>1$ was arbitrary, the lower
limit of the function $g(x)=f(x)x^{-1/n}$ equals to its upper limit
$\alpha_0$. This completes the proof of
Theorem \ref{gt}.

Theorem \ref{gt} and Lemma \ref{cru} immediately yield

\begin{theorem}\label{mnr}
For fixed $n$, the sequence $m(n,r)/r^n$ has a limit.
\end{theorem}

\section{List colorings}

Here we prove the choice version of Theorem \ref{mnr}.

\begin{theorem}\label{choicetheorem} For fixed integer $n>1$
the sequence $m_c(n,r)/r^n$ has a finite positive limit.
\end{theorem}

Denote by $f_c(N)$ the maximal possible
list chromatic number of an $n$-uniform 
hypergraph with $N$ edges. Since the list
chromatic number is always not less than
the chromatic number, we get 
\begin{equation}\label{snizu}
f_c(N)\geqslant \delta N^{1/n}
\end{equation}
for certain $\delta>0$ depending only on $n$.
Theorem \ref{choicetheorem} is equivalent to
the existence of a finite limit of
$f_c(N)/N^{1/n}$.

We use the following Chernoff
type concentration inequality
for the sum of independent 
$\{0,1\}$-valued random variables.

\begin{proposition}\label{LLN}
If $n$ is a positive
integer and $\xi_1,\ldots,\xi_n$ are 
independent random variables taking
values in $\{0,1\}$, $A$ is
the expectation of $S:=\sum_{i=1}^n \xi_i$,
$T\in [0,A]$, then 
$$
\prob \{S\leqslant A-T\}\leqslant e^{-\frac{T^2}{2A}}.
$$
\end{proposition}

See the proof, for example, in
\cite{michaelmitzenmacher2005}, Theorem 4.5.

We need the following technical statements.

\begin{lemma}\label{tl1}
Assume that $n>1$ is a fixed 
integer, $f:\mathbb{Z}_{\geqslant 0}\to \mathbb{R}_{>0}$
is a function satisfying 
\begin{equation}\label{t1}
    f(x)\leqslant \max_{a+b\leqslant x/2^{n-1}}f(a)+f(b)+M(f(a)^{\alpha}+f(b)^{\alpha}),\quad \forall x\geqslant x_0,
\end{equation}
for certain constants $x_0>0$, $\alpha\in (0,1)$,
$M>0$. Then $f(x)=O(x^{1/n})$ for large $x$.
\end{lemma}

\begin{proof}
We recursively define the increasing sequence $h_0\leqslant h_1\leqslant \ldots$
of positive numbers such that 
\begin{equation}\label{t2}
    f(x)\leqslant h_k\cdot x^{1/n}\quad \textrm{for}\quad
    1\leqslant x\leqslant x_0\cdot 2^{(n-1)k}.
\end{equation}
Choose $h_0$ large enough (so that $h_0>1$, 
\eqref{t2} for $k=0$
is satisfied and also something else, to be specified later,
holds). Assume that $k\geqslant 1$ and \eqref{t2}
holds for $0,1,\ldots,k-1$. Choose $x\in (x_0\cdot 2^{(n-1)(k-1)},x_0\cdot 2^{(n-1)k}]$.  
This $x$ satisfies \eqref{t1}.
Fix corresponding $a,b$ and consider two
cases: either $\min(a,b)=0$ or both $a,b$ are positive.

In the first case we get 
\begin{equation}\label{t3}
  f(x)\leqslant h_{k-1}(2^{1-n} x)^{1/n}+M h_{k-1}^{\alpha}(2^{1-n} x)^{\alpha/n})+f(0)+M(f(0))^\alpha.
\end{equation}
If $h_{k-1}$ is large enough, the right hand side
does not exceed $h_{k-1} x^{1/n}$. This may
be guaranteed by choosing large enough $h_0$. 

In the second case both $a$ and $b$ satisfy the induction
hyphotesis and we get

\begin{equation}\label{t4}
    f(x)\leqslant h_{k-1}(a^{1/n}+b^{1/n})+
    2M h_{k-1}^{\alpha}(2^{1-n}x)^{\alpha/n}\leqslant
    h_{k-1}x^{1/n}+2M h_{k-1}^{\alpha}(2^{1-n}x)^{\alpha/n}.
\end{equation}
The right hand side of \eqref{t4} does not exceed
$$h_{k-1}x^{1/n}\left(1+2M x^{(\alpha-1)/n})\right).
$$
Since $x\geqslant x_0\cdot 2^{(n-1)(k-1)}$,
it allows to choose
$$
h_k=h_{k-1}\left(1+2 M x_0^{(\alpha-1)/n}\cdot 2^{(\alpha-1)(k-1)(n-1)/n}
\right)
$$
and \eqref{t2} for $k$ holds.
The sequence $h_k$ obviously increases and
the sequence $h_k/h_{k-1}-1$ decays exponentially.
Thus the infinite product of $h_k/h_{k-1}$
converges, i.e., $h_k$ is bounded. Lemma is proved.
\end{proof}

\begin{lemma}\label{tl2}
Assume that $n>1$ is a fixed integer, $\alpha\in (0,1)$,
$M>0$ and $\delta>0$ are fixed constants. Then there
exist constants $C>0$ and $x_0>0$ such that
for $p=2$ and $p=3$ we have
\begin{equation}\label{t5}
   \delta\left(x^{1/n}-\sum_{i=1}^p a_i^{1/n}\right)\geqslant
C\left(x^{\alpha/n}-\sum_{i=1}^p a_i^{\alpha/n}\right)+
Mx^{\alpha/n},
\end{equation}
for every $x\geqslant x_0$ and $a_i\geqslant 0$ such that
\[
\sum_{i=1}^p a_i\leqslant p^{1-n}\cdot x. 
\]
\end{lemma}

\begin{proof}
The left hand side of \eqref{t5} is always non-negative
by Jensen inequality for the concave function $t^{1/n}$.
Note that if $a_i=p^{-n}x$ for all $i=1,\ldots,p$, 
then $x^{\alpha/n}-\sum_{i=1}^p a_i^{\alpha/n}=x^{\alpha/n}
(1-p^{1-\alpha})<0$. Fix $C$ such that
$C(2^{1-\alpha}-1)>M$. 

Then we may fix 
$\varepsilon>0$ such that whenever $|a_i/x-p^{-n}|<\varepsilon$
for all $i=1,\ldots,p$, the right hand side of
\eqref{t5} is non-positive and therefore
\eqref{t5} holds in this case. 

By Proposition \ref{Jensen_acc},
otherwise there exists $\varepsilon_1>0$
such that left hand side of \eqref{t5} is not less
than $\varepsilon_1 x^{1/n}$. It implies that \eqref{t5}
holds in this case for large enough $x$.
\end{proof}

\begin{corollary}\label{tl3}
Assume that $n>1$ is a fixed 
integer,
$f:\mathbb{Z}_{\geqslant 0}\to \mathbb{R}_{>0}$
is a function satisfying $f(x)\geqslant \delta x^{1/n}$
for all $x\geqslant 0$ and
\begin{equation}\label{t7}
    f(x)\leqslant \max_{a_1+\ldots+a_p\leqslant x/p^{n-1}}\sum f(a_i)+Mx^{\alpha/n},\quad \forall x\geqslant x_0
\end{equation}
for $p\in \{2,3\}$ and
certain constants $x_0>0$, $\alpha\in (0,1)$,
$M>0$. Then there exist $C>0$ and $x_1>0$ such that
the function $\tilde{f}(x):=f(x)+Cx^{\alpha/n}-\delta x^{1/n}$
satisfies 
\begin{equation}\label{t8}
    \tilde{f}(x)\leqslant \max_{a_1+\ldots+a_p\leqslant x/p^{n-1}}\sum \tilde{f}(a_i),\quad \forall x\geqslant x_1.
\end{equation}
\end{corollary}

Now we give a recursive estimate for the maximal
possible list chromatic number for an $n$-uniform
hypergraph
with prescribed number of edges.

\begin{lemma}\label{choice}
There exists a constant $M>0$ such that for $p\in \{2,3\}$
and all non-negative integer $N$
we have
$$
f_c(N)\leqslant \max_{a_1+\ldots+a_p\leqslant N/p^{n-1}}\sum_{i=1}^p f_c(a_i)+M(f_c(a_i))^{2/3}.
$$
\end{lemma}

\begin{proof}
Let $H=(V,E)$ be an $n$-uniform hypergraph with $|E|=N$. 
Assume that any vertex $v\in V$ edge has a list $L(v)$ 
consisting of 
$\sum_{i=1}^p f_c(a_i)+c_i$ admissible colors, where
\[
c_i :=\lfloor M(f_c(a_i))^{2/3}\rfloor.
\]
It suffices
to prove that $H$ has a proper list
colorings with these lists. 

As in the proof of Lemma \ref{cru},
we
partition $V$ onto
disjoints subsets $V_i$ 
so that the corresponding induced
subgraphs $H_i=(V_i,E_i)$ of $H$
satisfy $\sum |E_i|\leqslant N/p^{n-1}$.
Denote $a_i=|E_i|$.

For any color $\alpha$
choose $\xi(\alpha)\in \{1,\ldots,p\}$
independently at random with
probability of
$\{\xi(\alpha)=i\}$ proportional to
$f_c(a_i)+c_i$.
Call an edge $e\in E$ \textit{nice} if it either
contains the vertices from different $V_i$'s,
or $e\in E_i$ and 
$|L(v)\cap \xi^{-1}(i)|\geqslant f_c(a_i)$
for all $n$ vertices $v\in e$. 
Due to Proposition \ref{LLN} the probability that 
an edge $e\in E_i$
is not nice does not exceed
$$
n\exp\left(-\frac{c_i^2}{2(f_c(a_i)+c_i)}\right)
$$
(the multiple $n$ comes from the number of vertices
in $e$ and applying the union bound).

If we permanently denote $f_c(a)=x$ for 
non-negative integer $a$,
$y=\lfloor Mx^{2/3}\rfloor$,
then
using the lower bound \eqref{snizu} and assuming 
$M>100$
we conclude that 
$$
\frac{y^2}{2(y+x)}
\geqslant \frac{M^2x^{4/3}}
{10\max(x,Mx^{2/3})}=
\frac1{10}
\min\left(M^2x^{1/3},Mx^{2/3}\right)\geqslant \frac{Mx^{1/3}}{10}
\geqslant \frac{M\delta^{1/3}a^{1/(3n)}}{10},
$$
and 
$$
a\exp\left(-\frac{y^2}{2(x+y)}\right)<1/n
$$
for all $a=0,1,\ldots$ 
provided that the constant
$M$ is chosen
large enough.

Fix such
a value of $M$, then
\[
n\sum_{i=1}^p a_i\exp\left(-\frac{c_i^2}{2(f_c(a_i)+c_i)}\right) < 1
\]
and with positive probability all edges are
nice. This allows to properly color each $H_i$
using the colors only
from $\xi^{-1}(i)$ and get
a proper coloring of $H$.
\end{proof}

Now Lemma  \ref{tl1} and
Lemma \ref{choice} for $p=2$ yield
$f_c(x)=O(x^{1/n})$. Therefore $f_c$ satisfies the conditions
of Corollary \ref{tl3} for
$\alpha=2/3$ and certain $M>0$ (and $x_0=1$).
Corresponding function $\tilde{f_c}$ satisfies
the conditions of Theorem \ref{gt},
hence $f_c(x)/x^{1/n}$ has a finite limit
and Theorem \ref{choicetheorem} is proved.

\section{Further questions}

\begin{itemize}
    \item First, recall that the Erd{\H o}s conjecture is still open in the case $n = 3$.
The survey and the best current lower bound are given in~\cite{cherkashin2019Erdos}.

    \item A hypergraph is called \textit{simple} if every pair of edges shares at most 1 vertex.
    Let $s(n,r)$ be the minimal number of edges in a simple $n$-graph which has no proper $r$-coloring.
    It is known~\cite{kostochka2001chromatic} that for a fixed $n$ one has 
    \[
    cr^{2n-2}\ln r \leqslant s(n,r) \leqslant Cr^{2n-2}\ln r.
    \]
    Unfortunately, we cannot show regularity of $s(n,r)$.

    \item Also it is natural to ask if $m(n,r)$ is regular on the first variable, i.e.  
\[
\lim\limits_{n\to \infty} \frac{m(n+1,r)}{m(n,r)} = r?
\]
    
\end{itemize}

\paragraph{Acknowledgements.} The paper is supported by the Russian Scientific Foundation grant 17-71-20153.
We are grateful to Saint Petersburg State
University IMC team and Mikhail Antipov for checking the proof.
%and especially to Stanislav Ershov for computational help. 
We are grateful to Alexander Sidorenko for pointing out our inattention, relating to the use of Tur{\'a}n numbers.

\bibliographystyle{plain}
\bibliography{main}

\end{document}